\documentclass[11pt]{paper}

\usepackage{amssymb,amsfonts,amsthm,amsmath}
\usepackage{latexsym,multicol,graphicx}
\usepackage{rotating}

\usepackage{graphicx, verbatim}
\usepackage{tikz}
\usetikzlibrary{decorations.pathreplacing}
\usetikzlibrary{snakes}
\usepackage{xcolor}
\usepackage{verbatim}

\usepackage{pict2e}

\usepackage[a4paper, hmargin=2.5cm, vmargin={3cm, 3cm}]{geometry}

\usepackage{fancyhdr}
\usepackage[title,titletoc,toc]{appendix}

\newtheorem{theorem}{Theorem}
\newtheorem{lemma}[theorem]{Lemma}

\newtheorem{corollary}[theorem]{Corollary}

\newtheorem{claim}[theorem]{Claim}

\renewcommand{\phi}{\varphi}

\newcommand{\la}{\lambda}

\newcommand{\cE}{\mathbb E}

\newcommand{\cP}{\mathbb P}

\newcommand{\im}{{\rm i}}

\renewcommand{\le}{\leqslant}
\renewcommand{\ge}{\geqslant}

\newcommand{\bR}{\mathbb R}
\newcommand{\bC}{\mathbb C}
\newcommand{\bZ}{\mathbb Z}

\newcommand{\bN}{\mathbb N}

\newcommand{\bP}{\mathbb P}

\DeclareMathOperator*{\esssup}{ess\,sup}

\begin{document}

\title{Real zeroes of random polynomials, I \\ Flip-invariance, Tur\'an's lemma, and the Newton-Hadamard polygon}

\author{Ken S\"{o}ze \thanks{290W 232nd Str, Apt 4b, Bronx, NY 10463, USA;
\texttt{sozeken65@gmail.com}}}

\date{\today}

\maketitle

\hfill{To Ildar Ibragimov with admiration}

\begin{abstract}
We show that with high probability the number of real zeroes of a random polynomial is bounded by the number of vertices
on its Newton-Hadamard polygon times the cube of the logarithm of the polynomial degree. A similar estimate holds for zeroes
lying on any curve in the complex plane, which is the graph of a Lipschitz function in polar coordinates.
The proof is based on the classical Tur\'an lemma.
\end{abstract}

\section{Introduction}
This work is motivated by the following question attributed to Larry Shepp:
Let \[ P(z) =  \sum_{k=0}^n \la_k z^k \]
be a random polynomial of degree $n\ge 2$ with independent identically distributed random
coefficients $\la_k$. {\em Is it
true that the expected number of real zeroes of $P$ is bounded by $C \log n$}?
Since the classical work of Mark Kac~\cite{Kac}, for many ``decent'' distributions
of the coefficients, it has been proven by Erd\H{o}s and Offord~\cite{EO}, Logan and Shepp~\cite{LS},
Ibragimov and Maslova~\cite{IM, IM-Doklady}, Shepp and Farahmand~\cite{SF} (by no means is this list complete).
Here, we are interested in a bound valid for all distributions.
Several years ago, Ibragimov and Zaporozhets~\cite{IZ} proved that for any distribution of the coefficients,
the expected number of real zeroes is $o(n)$ as $n\to\infty$. Later, in works that remained unpublished,
this was independently improved  by Kabluchko and Zaporozhets and by Krishnapur and Zeitouni to $O(\sqrt{n})$.
In the opposite direction, Zaporozhets~\cite{Zaporozhets} constructed an example of a distribution
wherein the mean number of real zeroes remains bounded as $n\to\infty$.

In this work we suggest two approaches to this question. The first one, presented in this part, is
based on tools from harmonic and complex analysis (Tur\'an's lemma and Jensen's formula).
In the case when the coefficients of $P$ are independent and identically distributed,
it gives a bound $C\, \log^4 n$, which is weaker than the estimate we prove in Part~II.
On the other hand, the approach of Part~I needs less restrictive condition (which we call
``the property $(\Theta)$'') than
independence and identical distribution of the coefficients.
Assuming the property $(\Theta)$, we show that, with
high probability, the number of real zeroes of $P$ is bounded by $C\, V(P)\, \log^3n$ where $V(P)$
is the number of vertices on the Newton-Hadamard polygon of $P$. It also gives the same upper
bound for the number of zeroes of $P$ on
any curve in the complex plane, which is the graph of a Lipschitz function in polar coordinates.

The second approach, which we shall present in Part~II, is based on Descartes' rule of sign changes and on a
new anti-concentration estimate for random permutations
of large order, which might be of independent interest. Both approaches
may be viewed as further development of the techniques introduced
in the pioneering work of Littlewood and Offord~\cite{LO}.

\section{Main results}

\subsection{Key definitions}
We start with three definitions needed to formulate our results.
In what follows, $P$ always stands for
a polynomial of degree $n$ with,
generally speaking, complex-valued  coefficients $\la_k$.

\subsubsection{The number of vertices on the Newton-Hadamard polygon}\label{subsubsub:Newton-Hadamard}
We denote by $V(P)$ the number of vertices on the graph
of the convex polygonal function
\[
t \mapsto \max_{0\le k \le n} \bigl( \log|\la_k| + kt \bigr), \qquad t\in\bR.
\]
Although we will not use it,
it is not difficult to see that equivalently $V(P)$ can be defined as the number of vertices
on the Newton-Hadamard polygon, which is the the upper envelope of convex functions $\phi$ such that
$ \phi (k) \le - \log |\la_k|$, $0\le k \le n$ (in other words, the lower boundary of the
convex hull of $n+1$ vertical rays $\bigl\{ (k, y)\colon - \log |\la_k|\le y < +\infty, \ 0\le k \le n \bigr\} $).
For more on this, see~\cite[Chapter~IV, Problem~41]{PolyaSzego}.

\subsubsection{The Lipschitz curves}
By $\Gamma$ we denote an arbitrary curve defined in polar coordinates
by
\[
\Gamma = \bigl\{z=re^{\im\theta}\colon \theta=\theta(r), 0\le r<\infty \bigr\}\,.
\]
If
\[
\bigl| \theta (r_1) - \theta (r_2) \bigr| \le L\, \bigl| \log \frac{r_1}{r_2}\bigr|\,,
\]
then we call $\Gamma$ an $L$-Lipschitz curve. We denote by
$ N(\Gamma; P)$ the number of zeroes of $P$ on $\Gamma$ (counted with multiplicities).

\subsubsection{Flips of the coefficients}
Let $\la'$ and $\la''$ be $\bC^{n+1}$-valued random variables defined on the same probability space
and having the same distribution.
For $k\in\{0, 1, ..., n\}$, we put $\la_k^+=\la_k'$ and $\la_k^-=\la_k''$, and then,
for any $(n+1)$-tuple of signs $\sigma\in\{+,-\}^{n+1}$, let
$\la^\sigma=(\la_0^{\sigma_0}, \la_1^{\sigma_1}, ... , \la_n^{\sigma_n})$.
We say that the joint law of $\la'$ and $\la''$ is {\em flip-invariant} if
the random variables $\bigl\{ \la^\sigma \bigr\}_{\sigma\in\{+,-\}^{n+1}}$ are equidistributed.

\subsubsection{The property $(\Theta)$}
Here, we introduce our assumption on the distribution of the coefficients $\la\in\bC^{n+1}$
of the polynomial $P$.
We say that the coefficients of the random polynomial $P$ possess property $(\Theta)$
if there exist random variables $\la'$ and $\la''$ equidistributed with $\la$ whose joint law is flip-invariant
and such that, for some $a\in\bC$ and for each $k\in \{0, 1, ..., n\}$,
\[
\bigl| \la_k^{\sigma_k} - \lambda_k^{-\sigma_k} \bigr| \ge \frac12 \bigl[ \bigl| \la_k^{\sigma_k} - a \bigr| + \bigl| \la_k^{-\sigma_k} - a \bigr| \bigr]
\qquad {\rm a.s.\,}.
\]

\medskip
Note that for our purposes, it would suffice to have this inequality with any
constant $\kappa>0$ instead of $\tfrac12$. In the examples, which we  will bring below,
this condition holds with the value $\kappa=\tfrac12$.
To simplify our notation, we decided to fix this value of $\kappa$.

\subsubsection{Three examples of distributions with property $(\Theta)$}

\paragraph{Symmetric distributions.}
For $k\in \{0, 1, ..., n\}$, denote by $\tau_k\colon \bC^{n+1}\to\bC^{n+1}$ the map, which
maps $w_k\mapsto -w_k$ and keeps fixed the rest of coordinates of $w\in\bC^{n+1}$.
Suppose that, for each $k\in \{0, 1, ...\,, n\}$, $\tau_k \circ \la$ has the same distribution
as $\la$. Then the distribution of $\la$ has property $(\Theta)$ with $a=0$, $\la'=\la$, and
$\la'' = -\la$.

Note that in this example we do not assume independence of $\la_k$'s.

\paragraph{Complex-valued independent identically distributed random variables
$\la_0$, $\la_1$, ..., $\la_n$.}
Denote by $\zeta$ the common distribution of $\la_k$'s. We need to produce two random variables
$\zeta^{\pm}$ having the same distribution as $\zeta$ and such that, for some $a\in\bC$,
\[
\bigl| \zeta^+ - \zeta^- \bigr| \ge \frac12 \bigl[ \bigl| \zeta^+ - a \bigr| + \bigl| \zeta^- - a \bigr| \bigr].
\]
We first assume that the probability space $\Omega$ is a union of $2N$ atoms $\omega_i$
having the same probability $\tfrac1{2N}$. Then the general case will follow by approximation\footnote{
Indeed, take a sequence of random variables  $( \zeta_N )$ that converges in distribution to $\zeta$ and
such that $\zeta_N$ attains $2N$ values (not necessarily distinct ones)
with probability $\tfrac1{2N}$ each. Let $(\zeta_N^+, \zeta_N^-)$ be a pair of random variables defined on the same
probability space as $\zeta_N$, equidistributed with $\zeta_N$ and such that,
for some $a_N\in\bC$,
\[
\bigl| \zeta_N^+ - \zeta_N^- \bigr| \ge \frac12 \bigl[ \bigl| \zeta_N^+ - a_N \bigr| + \bigl| \zeta_N^- - a_N \bigr| \bigr].
\eqno (*)
\]
Since $\zeta_N$ converge to $\zeta$ in distribution, the sequence of laws of $\zeta_N$ is tight. Then the sequence
of joint laws of pairs $(\zeta_N^+, \zeta_N^-)$ is tight as well, and we can choose a subsequence $(\zeta_{N_j}^+, \zeta_{N_j}^-)$
that converges in distribution to a pair of
random variables $(\zeta^+, \zeta^-)$ defined on the same probability space as $\zeta$ and equidistributed with $\zeta$.

\smallskip
Furthemore, by tightness of the sequence of laws of $(\zeta_N^+, \zeta_N^-)$, we can choose a large positive
constant $L$ so that, for every $N$, $\bP\bigl\{ |\zeta_N^+ - \zeta_N^-| > L \bigr\} < \tfrac12 $. Therefore,
on an event of probability at least $\tfrac12$,
\[
|a_N| \le |\zeta_N^+ - a_N| + | \zeta_N^- - a_N | + | \zeta_N^+ - \zeta_N^- | \le 3| \zeta_N^+ - \zeta_N^- | \le 3L\,.
\]
Since both $a_N$ and $L$ are non-random, it follows that $|a_N|\le 3L$.
Then, extracting  from $(a_{N_j})$ a convergent subsequence, denoting by $a$ its limit, and applying ($*$) with $N=N_{j}$, $j\to\infty$,
we get the result.
}.

Let
\[
d = \max_{1\le i < j \le 2N} |\zeta(\omega_i)-\zeta(\omega_j)|
\]
be the diameter of the point configuration $\bigl\{ \zeta(\omega_1), ... , \zeta(\omega_{2N}) \bigr\}$ in $\bC$.
Pick up from this configuration a pair of points with the maximal distance. Without loss of generality, assume that they correspond to
the atoms $\omega_{2N}$ and $\omega_{2N-1}$, that is,
$ d = \bigl| \zeta(\omega_{2N-1}) - \zeta (\omega_{2N}) \bigr| $.
Then consider the remaining point configuration and repeat the procedure.
At the last $N$-th step we are left with two points $\zeta(\omega_1)$ and $\zeta(\omega_2)$. Then denote
by $a$ the center of the line segment that connects these two points, that is,
$ a = \tfrac12 \left( \zeta(\omega_1) + \zeta(\omega_2) \right) $.

By construction, for each $1\le i \le N$,
the point $a$ lies at distance at most
$ | \zeta(\omega_{2i-1}) - \zeta (\omega_{2i})| $ from each of the two points
$\zeta (\omega_{2i-1})$, $\zeta(\omega_{2i})$.
Hence,
\[
\bigl| \zeta(\omega_{2i-1}) - \zeta (\omega_{2i}) \bigr|
\ge \frac12 \bigl[ \bigl| \zeta (\omega_{2i-1}) - a \bigr| + \bigl| \zeta (\omega_{2i}) - a \bigr| \bigr].
\]
It remains to let $\zeta^+=\zeta$, and
\[
\zeta^- (\omega_{2i-1}) = \zeta (\omega_{2i})\,, \quad
\zeta^- (\omega_{2i}) = \zeta (\omega_{2i-1})\,, \qquad
1 \le i \le N\,.
\]

\paragraph{Real-valued independent
random variables $\la_0$, $\la_1$, ..., $\la_n$ which have a common median.}
Arguing similarly to the previous example, we construct the coefficients $\la_k^\pm$ equidistributed with $\la_k$
and satisfying $\lambda^+_k + \lambda^-_k = 2a$, $0\le k \le n$, where $a$ is the common median for
$\la_0$, $\la_1$, \ldots \, $\la_n$.

Note that in this example we have not assumed that
the coefficients $\la_k$ are identically distributed.

\subsubsection{A technical assumption}
To avoid degeneracies, in what follows, we always assume that the coefficients of the random polynomial $P$ satisfy
\begin{equation}\label{eq:nondegen-lambda}
\cP \bigl\{ \la_0 = 0 \bigr\} = \cP \bigl\{ \la_n = 0 \bigr\} =0\,.
\end{equation}
That is, $P$ does not vanish at the origin and the degree of $P$ does not drop.
This condition can be dropped at the cost of a somewhat longer wording of the main result.

\subsection{The main theorem}
At last, we are ready to state the main result of this note:

\begin{theorem}\label{thm:main}
Let $P$ be a random polynomial of degree $n\ge 2$ with coefficients having the property $(\Theta)$ and
satisfying the non-degeneracy condition~\eqref{eq:nondegen-lambda}. Let  $L\ge 0$ and $A>0$.
Then, with probability at least $1-n^{-A}$, we have
\[
\sup\bigl\{N(\Gamma; P)\colon \Gamma \ {\rm is\ } L\!-\!{\rm Lipschitz} \bigr\} \le C(A, L)\, V(P)\, \log^3 n\,.
\]
\end{theorem}
\noindent Here, $C(A, L)$ is a positive value that depends only on the parameters $A$ and $L$.

\medskip
Note that there is no hope for a similarly strong non-random estimate: a construction, which goes back
to  Bloch and P\'olya~\cite{BP}, allows one to construct a polynomial $P$ of any degree
$n\ge 2$ with $V(P)=2$ and with at least $\sqrt{ n/\log n }$ positive zeroes.

\subsection{A corollary for the case of i.i.d. coefficients}
As an almost immediate corollary, we obtain

\begin{corollary}
Suppose that the coefficients of $P$ are independent identically distributed random variables
satisfying the non-degeneracy condition~\eqref{eq:nondegen-lambda}. Then
\[
\sup\bigl\{ \cE [ N(\Gamma; P) ]\colon  \Gamma \ {\rm is\ } L\!-\!{\rm Lipschitz}\bigr\} \le C(L)\, \log^4 n, \qquad n\ge 2\,.
\]
\end{corollary}
\noindent In particular,
$ \cE [ N(\bR; P) ] \le C\, \log^4 n $
with a positive numerical constant $C$. As we have already mentioned, the latter estimate will be improved in
Part~II by a different technique.

\begin{proof} We use Theorem~\ref{thm:main} with $A=1$. Since the total number of zeroes of
$P$ on $\Gamma$ cannot exceed $n$, a set of probability $n^{-1}$ can contribute to the expectation
$\cE [ N(\Gamma; P) ]$ by at most~$1$. Therefore,
\[
\sup\bigl\{ \cE [ N(\Gamma; P) ]\colon  \Gamma \ {\rm is\ } L\!-\!{\rm Lipschitz}\bigr\}
\le C(L) \log^3 n \cdot \cE [ V(P)]\,.
\]
To estimate the mean $\cE [V(P)]$, we note that $V(P)$
equals the cardinality of the set of indices $\nu\in \{0, 1, ... , n\}$ such that, for some $r\in (0, \infty)$,
\begin{equation}\label{eq:max-index}
\nu \ {\rm is\ the\ largest\ index\ satisfying\ } \quad |\la_\nu| r^\nu = \max_{0\le k \le n} |\la_k|r^k\,,
\end{equation}
If~\eqref{eq:max-index} holds for some $r\in (0, 1)$, then
\[
|\la_\nu| > |\la_k| \qquad {\rm for\ each\ } k\in\{0, 1, ..., \nu-1\}\,.
\]
By symmetry, the probability of this event
does not exceed $\tfrac1{\nu+1}$. Similarly,
if~\eqref{eq:max-index} holds for some $r\in [1, \infty)$, then
\[
|\la_\nu| > |\la_k| \qquad {\rm for\ each\ } k\in\{\nu+1, \nu+2, ... , n\}\,,
\]
and the probability of this event is $\le\tfrac{1}{n-\nu+1}$. Thus,
\[
\cE [V(P)] \le 2 \bigl( 1 + \tfrac12 +\, ...\, + \tfrac1{n} \bigr) \le C \log n\,,
\qquad n\ge 2\,,
\]
proving the corollary.
\end{proof}

\subsection{A probabilistic lower bound for a random polynomial on an arc}

The following result is the main tool needed for the proof of Theorem~\ref{thm:main}.
Likely, it may be of independent interest. Put
\[
S(r, P) = \sum_{k=0}^n |\la_k| r^k\,.
\]
\begin{theorem}\label{thm:lower-bound}
Let $P$ be a random polynomial of degree $n\ge 2$ with coefficients having the property $(\Theta)$.
Let $m\in\bN$, let $r>0$, and let
$I\subset \bR$ be an interval of length at most $2\pi$. Then, for some positive numerical constant $c$,
\[
\cP \Bigl\{  \max_{\theta\in I} |P(re^{\im\theta})| \le n^{-2} \bigl(  c\,|I| \bigr)^{6m} \, S(r, P) \Bigr\} \le 2^{-m}\,.
\]
\end{theorem}

The proof of this theorem will be given in Section~\ref{Section3}.

\subsection{The reduction principle}\label{subsect:reduction}
Our starting point is the following claim:
\begin{lemma} Suppose that the coefficients of the random polynomial $P$ possess the property $(\Theta)$.
Then, for any Borel set $\Lambda\subset\bC^{n+1}$, we have
\[
\cP\bigl\{ \lambda\in\Lambda \bigr\}
\le \sup\, \cP\bigl\{ \upsilon\in\Lambda \bigr\}\,,
\]
where the supremum is taken over all random variables $\upsilon\colon  \{+, -\}^{n+1}\to\bC^{n+1}$ of the form
$\upsilon^\sigma = (\upsilon_0^{\sigma_0}, \upsilon_1^{\sigma_1}, \, ...\,, \upsilon_n^{\sigma_n})$ such that
the random variables $\upsilon_k$ are independent, take the values $\upsilon^{\pm}_k$ with probability $\tfrac12$ and,
for some $a\in\bC$,
\begin{equation}\label{eq:upsilon}
\bigl| \upsilon_k^+ - \upsilon_k^- \bigr| \ge \frac12 \bigl[ \bigl| \upsilon_k^+ - a \bigr| + \bigl| \upsilon_k^- - a \bigr| \bigr]\,,
\qquad k\in\{0, 1, ..., n \}\,.
\end{equation}
\end{lemma}

It is worth noting that for independent real-valued random variables, this reduction was
used already by Kolmogorov in~\cite{Kolm}, where he proved a slightly weaker version of what is called nowadays the
Kolmogorov-Rogozin concentration inequality.

\begin{proof}
Let $\Omega$ be the underlying probability space of $\la'$ and $\la''$ in the definition of flip-invariance, let
$\widetilde\Omega=\Omega\times\{ +, - \}^{n+1}$ be the product space with the uniform distribution over all sign sequences
$\sigma=(\sigma_0, \sigma_1, ...\,, \sigma_n)$, and let
$\la^\sigma \stackrel{\rm def}=(\la_0^{\sigma_0}, \la_1^{\sigma_1}, ...\,, \la_n^{\sigma_n})$. Then $\la^\sigma\colon \widetilde{\Omega}\to \bC^{n+1}$
and, for each $\sigma\in \{ +, - \}^{n+1}$, the random variables $\la^\sigma$ and $\la$ are equidistributed.
Therefore,
\[
\bP^\Omega \{ \la\in\Lambda \} = \bP^{\Omega\times \{+,-\}^{n+1}} \{ \la^\sigma \in\Lambda \}
\le \esssup_{\omega\in\Omega} \bP^{\{+,-\}^{n+1}} \{ \la^\sigma (\omega) \in\Lambda \}\,.
\]
It remains to observe that, for a.e. $\omega\in\Omega$, the random variable $\upsilon^\sigma = \la^\sigma(\omega)$ satisfies~\eqref{eq:upsilon}
with the same value $a$ as in the condition $(\Theta)$. Hence, the essential supremum on the RHS does not exceed the supremum
in the conclusion of the lemma.
\end{proof}

Thus, it suffices to prove Theorems~\ref{thm:main} and~\ref{thm:lower-bound}
for a  special class of random polynomials. Hence,
in what follows, we assume that:
\begin{itemize}
\item[(a)]
the underlying probability space is $\{ +, -\}^{n+1}$
with the uniform distribution over sign sequences,
and, as above, we denote the elements of this space by $\sigma= (\sigma_0, \sigma_1, ... , \sigma_n)$;
\item[(b)]
$( \la_k^\pm )$ are $2n+2$ complex numbers, $a$ is a complex number, and for each $k\in\{0, 1, ..., n \}$,
\[
\bigl| \la_k^+ - \lambda_k^- \bigr| \ge \frac12 \bigl[ \bigl| \la_k^+ - a \bigr| + \bigl| \la_k^- - a \bigr| \bigr]\,;
\]
\item[(c)]
the random variables $\la_k$ are independent and $\la_k$ takes the values $\la^\pm_k$ with probability $\tfrac12$ each.
\end{itemize}

\section{Proof of Theorem~\ref{thm:lower-bound}}\label{Section3}

The main ingredient of the proof of Theorem~\ref{thm:lower-bound} is Tur\'an's lemma~\cite[Chapter~5, Lemma~1]{Mont}
(see also~\cite[Chapter~1]{Nazarov}):
\begin{lemma}\label{lemma:Turan}
Let
\[
p(t) = \sum_{k=1}^m a_k e^{\im \ell_k t}, \qquad a_k\in\bC, \quad \ell_k\in\bZ, \ell_k\ne\ell_j {\ \rm for\ } k\ne j\,.
\]
Then for every interval $I\subset\bR$ of length at most $2\pi$,
\[
\max_I |p| \ge \bigl( b|I| \bigr)^{m-1}\, \sum_{k=1}^m |a_k|
\]
with a positive numerical constant $b$.
\end{lemma}
Note that the conclusion of this lemma is usually stated in the form
\[
\max_I |p| \ge \bigl( b|I| \bigr)^{m-1}\, \max_{[-\pi, \pi]} |p|\,.
\]
Since
\[
\sum_{k=1}^m |a_k|
\le \sqrt{m}\, \Bigl( \sum_{k=1}^m |a_k|^2 \Bigr)^{1/2}
= \sqrt{m}\, \Bigl( \frac1{2\pi}\, \int_{-\pi}^\pi |p(t)|^2\, {\rm d}t  \Bigr)^{1/2}
\le \sqrt{m}\, \max_{[-\pi, \pi]} |p|\,,
\]
the version we will be using readily follows from the usual one.

\subsection{The case of few large coefficients}
{\em Given $m\in\bN$ and $r>0$, we assume that, for some $a\in\bC$,
\begin{equation}\label{eq:100}
\# \bigl\{ k\colon |\la_k-a| r^k \ge \delta S(r; P) \bigr\} \le 2m
\end{equation}
and show that for every interval $I\subset \bR$ of length at most $2\pi$,
\begin{equation}\label{eq:large-coeff}
\max_{\theta\in I} \bigl| P(re^{\im\theta} )\bigr| \ge cn^{-1}\, \bigl( b\, |I| \bigr)^{4m+1}\, S(r; P)
\end{equation}
provided  that $\delta= c_1 n^{-2}\, \bigl( b\, |I| \bigr)^{4m+1} $ with a sufficiently small constant
$c_1$}.

\subsubsection{The polynomial $\bar P$}
Put $\bar P (z) = (1-z) P(z)$. We need this polynomial to get rid of the dependence on the value
of $a$. Note that when $a=0$ this polynomial is not needed.

\begin{claim}\label{cl:Pbar}
\begin{equation}\label{eq:claim5}
S(r; \bar P) \ge \frac{1+r}{2(n+1)} S(r; P)\,, \qquad 0< r<\infty\,.
\end{equation}
\end{claim}

\begin{proof}
First, assume that $0<r\le 1$. Then
\begin{eqnarray*}
S(r; \bar P) &=& |\la_0| + \sum_{k=1}^n | \la_k - \la_{k-1} | r^k + |\la_n| r^{n+1} \\
&\ge& \frac1{n+1} \Bigl[ (n+1)|\la_0| + \sum_{k=1}^n (n+1-k) |\la_k-\la_{k-1} | r^k \Bigr] \\
&\ge& \frac1{n+1} \Bigl[ (n+1)|\la_0| + \sum_{k=1}^n (n+1-k) |\la_k| r^k
- \sum_{k=1}^n (n+1-k) |\la_{k-1}| r^k \Bigr] \\
&=& \frac1{n+1} \Bigl[ (n+1)|\la_0| + \sum_{k=1}^{n-1} \bigl( (n+1-k) - (n-k) r \bigr) |\la_k| r^k
+ | \la_n| r^n - n |\la_0| r \Bigr] \\
&=& \frac1{n+1} \Bigl[ \sum_{k=0}^{n-1} \bigl( (n+1-k) - (n-k) r \bigr) |\la_k | r^k + | \la_n| r^n \Bigr]\,.
\end{eqnarray*}
For $0<r\le 1$, we have $(n+1-k) - (n-k)r \ge 1$. Thus, the RHS of the previous estimate is
\[
\ge \frac1{n+1} \sum_{k=0}^n |\la_k|r^k \stackrel{0<r\le 1}\ge \frac{1+r}{2(n+1)} S(r; P)\,.
\]

Now, let $1\le r < \infty$. Then
\begin{eqnarray*}
S(r; \bar P) &=& |\la_0| + \sum_{k=1}^n | \la_k - \la_{k-1} | r^k + |\la_n| r^{n+1} \\
&\ge& \frac1{n+1} \Bigl[ \sum_{k=1}^n k \bigl( |\la_{k-1}| r^k - |\la_k| r^k \bigr)
+ (n+1) |\la_n| r^{n+1} \Bigr] \\
&=& \frac1{n+1} \sum_{k=0}^n \bigl( (k+1)r - k \bigr) |\la_k|r^k \,.
\end{eqnarray*}
Since $r\ge 1$, we have $(k+1)r - k \ge r$, and therefore, the RHS of the previous estimate is
\[
\ge \frac{r}{n+1} \sum_{k=0}^n |\la_k|r^k \stackrel{r\ge 1}\ge \frac{1+r}{2(n+1)} S(r; P)\,,
\]
proving the claim.
\end{proof}

\subsubsection{Proof of the lower bound~\eqref{eq:large-coeff} assuming~\eqref{eq:100}}

First, we observe that
\begin{equation}\label{eq:110}
\# \bigl\{ k\colon |\bar\la_k|r^k \ge 2\delta (1+r)S(r; P) \bigr\} \le 4m+2\,,
\end{equation}
where $\bar\la_k$ are coefficients of the polynomial $\bar P$.
Indeed,
\[
\bar P(z) = \la_0 + \sum_{k=1}^n (\la_k - \la_{k-1})z^k - \la_n z^{n+1}\,.
\]
Suppose that for some $k\in\bigl\{1,\, ...\,, n\bigr\}$ and $\delta>0$,
\[
|\la_k -\la_{k-1}| r^k \ge 2\delta (1+r) S(r; P)\,.
\]
Then
\[
|\la_k -  a|r^k  + |\la_{k-1} -  a| r^k \ge 2\delta (1+r) S(r; P)\,.
\]
That is, at least one of the following estimates holds: either
$ |\la_k -  a|r^k \ge \delta (1+r) S(r; P)$, or
$ |\la_{k-1} - a|r^k \ge \delta (1+r) S(r; P)$, proving~\eqref{eq:110}.

\medskip
Now, we split the polynomial $\bar P$ into large and small parts.
The small part $\bar P_{\tt sm}$ will consists of the terms $\bar\la_k r^k$ with
\[
|\bar\la_k|r^k \le 2\delta (1+r) S(r; P)\,.
\]
The rest goes to the large part $\bar P_{\tt la}$, which is a sum of at most $4m+2$ terms.
Using Tur\'an's lemma, we get
\begin{eqnarray*}
\max_{\theta\in I} \bigl| P(re^{\im\theta}) \bigr|
&\ge& (1+r)^{-1}  \max_{\theta\in I} \bigl| \bar P(re^{\im\theta}) \bigr| \\
&\ge& (1+r)^{-1} \Bigl[ \max_{\theta\in I} \bigl| \bar P_{\tt la}(re^{\im\theta}) \bigr|
-  \max_{\theta\in I} \bigl| \bar P_{\tt sm}(re^{\im\theta}) \bigr|\Bigr] \\
&\ge& (1+r)^{-1} \Bigl[ (b\,|I|)^{4m+1} S(r; \bar P_{\tt la}) - S(r; \bar P_{\tt sm}) \Bigr] \\
&\ge& (1+r)^{-1} \Bigl[ (b\,|I|)^{4m+1} \bigl( S(r; \bar P) - (n+1)\, 2\delta (1+r) S(r; P) \bigr) \\
&& \qquad \qquad \qquad \qquad \qquad \qquad   - (n+1)\, 2\delta (1+r) S(r; P) \Bigr] \\
&\stackrel{\eqref{eq:claim5}}\ge&
\Bigl[ (b\,|I|)^{4m+1} \Bigl( \frac1{2(n+1)} - 2(n+1)\delta\Bigr) - 2(n+1)\delta \Bigr]\, S(r; P) \\
&\ge& \Bigl[ (b\, |I|)^{4m+1} \Bigl( \frac1{4n} - 4n\delta \Bigr) - 4n\delta \Bigr]\, S(r; P)\,.
\end{eqnarray*}
Choosing $\delta = c_1 n^{-2} \bigl( b\, |I| \bigr)^{4m+1}$ with a sufficiently small constant $c_1$, we see that
the RHS of the previous estimate is
$ \ge c_2 n^{-1} \bigl( b\,|I| \bigr)^{4m+1}\, S(r; P) $, proving~\eqref{eq:large-coeff}.
\hfill $\Box$

\subsection{The dangerous configurations are rare}
Fix an interval $I\subset\bR$ of length at most $2\pi$ and fix $\delta$ as above.
Taking into account what we have just proven,
we see that in order to prove Theorem~\ref{thm:lower-bound},
we need to estimate the number of sign sequences $\sigma\in \{+,-\}^{n+1}$ such that
{\em there exist at least $2m$  $(\le n)$ indices
$k$ satisfying
\begin{equation}\label{eq:120}
\bigl| \la_k^{\sigma_k} -a \bigr| r^k \ge \delta S(r; P)
\end{equation}
but still
\begin{equation}\label{eq:125}
\max_{\theta\in I} \bigl| P(re^{\im\theta})\bigr| \le \delta_1 S(r; P)
\end{equation}
with some positive $\delta_1 \ll \delta $ to be chosen momentarily}.
We call the corresponding sequence of signs $\sigma$ {\em dangerous} and aim to show that
the number of dangerous sequences does not exceed $2^{n+1-m}$.

Take any dangerous sign sequence $\sigma$ and an $m$-element subset of the set of ``large coefficients''
that appear in condition~\eqref{eq:120}, and flip all the signs $\sigma_k$
corresponding to this $m$-element subset. Running over all possible $m$-elements subsets
of the set of ``large coefficients'' of a given dangerous sign sequence $\sigma$, we obtain at least
$ {2m \choose m} \ge 2^m $ different sign sequences.
Claim~\ref{claim:dangerous} (given few lines below) will yield that
{\em all new sign sequences obtained from all dangerous sign sequences $\sigma$ are different,
provided that the parameter $\delta_1$ is chosen as}
\begin{equation}\label{eq:130}
\delta_1 = \tfrac14 \delta\, \bigl( b\, |I| \bigr)^{2m-1}\,.
\end{equation}
Therefore, with the choice of the parameters as in~\eqref{eq:130}, {\em the total number of all dangerous sign sequences multiplied by  $ 2m \choose m $ cannot
exceed} $2^{n+1}$. At the same time, for any non-dangerous $\sigma$, we automatically have
\[
\max_{\theta\in I} \bigl| P(re^{\im\theta}) \bigr| >
\frac{c_1}4\, n^{-2}\, \bigl( b\, |I| \bigr)^{4m+1} \cdot \bigl( b\, |I| \bigr)^{2m-1}\, S(r; P)
> n^{-2} \bigl( c\, |I| \bigr)^{6m}\, S(r; P)\,.
\]
Therefore, Theorem~\ref{thm:lower-bound} follows if we prove the following claim:

\begin{claim}\label{claim:dangerous}
Let $\sigma\in\{+,-\}^{n+1}$ be any sign sequence. Suppose that there exist
two different $m$-element subsets $U_1, U_2 \subset \{0, 1, 2, ...\, n\}$, $U_1\ne U_2$, so that
the sets of flips corresponding to $U_1$ and $U_2$ turn $\sigma$ into a dangerous sign sequence with all
coefficients corresponding to flipped signs becoming ``large'' as in condition~\eqref{eq:120}.
Then the parameter $\delta_1$ cannot be as small as in~\eqref{eq:130}.
\end{claim}

\begin{proof} Once again, we will rely on Tur\'an's lemma.
We fix the sign sequence $\sigma$ and
denote by $\sigma^1$, $\sigma^2$ the flipped sign sequences, i.e.,
\[
\sigma_k^j =
\begin{cases}
-\sigma_k \quad &{\rm for\ } k\in U_j, \\
\sigma_k \quad &{\rm for\ } k\notin U_j.
\end{cases}
\]
By $P_j$, $j=1, 2$, we denote
the corresponding polynomials.
Choosing $k_1\in U_1\setminus U_2$ and $k_2\in U_2\setminus U_1$, we have
\[
\max_{\theta\in I} \bigl| P_j (re^{\im\theta})\bigr| \stackrel{\eqref{eq:125}}\le \delta_1 S(r; P_j)
\stackrel{\eqref{eq:120}}\le \frac{\delta_1}{\delta}\, \bigl| \la_{k_j}^{\sigma_{k_j}^j} - a \bigr|\, r^{k_j}\,, \qquad j=1, 2\,.
\]
Therefore,
\begin{equation}\label{eq:140}
\max_{\theta\in I} \bigl| ( P_1 - P_2)(re^{\im\theta}) \bigr|
\le \frac{\delta_1}{\delta}\,
\Bigl[ \bigl| \la_{k_1}^{\sigma_{k_1}^1} - a \bigr|\, r^{k_1}
+ \bigl| \la_{k_2}^{\sigma_{k_2}^2} - a \bigr|\, r^{k_2}
\Bigr]\,.
\end{equation}

On the other hand, the difference $P_1-P_2$ has at most $2m$ terms:
\[
\bigl( P_1-P_2 \bigr)(z)
= \sum_{k\in U_1 \bigtriangleup U_2} \bigl( \la_k^{\sigma_k^1} - \la_k^{\sigma_k^2}\bigr) z^k
\]
where, as usual, $ \bigtriangleup $ denotes the symmetric difference. For $k\in U_1 \bigtriangleup U_2$,
we have $\sigma_k^2 = - \sigma_k^1$. Then, by assumption~(b) in Section~\ref{subsect:reduction},
\[
\bigl| \la_k^{\sigma_k^1} - \la_k^{\sigma_k^2} \bigr|
\ge \frac12\, \Bigl[ \bigl| \la_k^{\sigma_k^1} - a \bigl| + \bigl| \la_k^{\sigma_k^2} - a \bigr| \Bigr]\,,
\qquad k\in U_1 \bigtriangleup U_2\,.
\]
In particular, this holds for $k=k_1, k_2$. Therefore, the RHS of~\eqref{eq:140} is
\[
\le \frac{2\delta_1}{\delta}\,
\Bigl[ \bigl| \la_{k_1}^{\sigma_{k_1}^1} - \la_{k_1}^{\sigma_{k_1}^2} \bigr|\, r^{k_1}
+ \bigl| \la_{k_2}^{\sigma_{k_2}^1} - \la_{k_2}^{\sigma_{k_2}^2} \bigr|\, r^{k_2} \Bigr]
\le \frac{2\delta_1}{\delta}\, S(r; P_1 - P_2)\,.
\]
If $\delta_1$ is as small as in~\eqref{eq:130}, this contradicts Tur\'an's lemma applied to
$P_1-P_2$. This proves the claim and finishes off the proof of Theorem~\ref{thm:lower-bound}.
\end{proof}

\section{Proof of Theorem~\ref{thm:main}}

\subsection{Preliminaries}

\subsubsection{}
First, we observe that it suffices to prove Theorem~\ref{thm:main} only for
zeroes of $P$ lying in the closed unit disk $\{ |z|\le 1 \}$.
To get the result for the rest of the zeroes, all one needs is to consider the polynomial
$P^*(z) = z^n P(z^{-1})$.

\subsubsection{}
It will be convenient to make the exponential change of variable
$z=e^{-2\pi w}$, $w=t+\im s$, $0\le t < \infty$,
and to deal with the exponential polynomial
\[
Q(w) = P(z) = \sum_{k=0}^n \la_k  e^{-2\pi kw}\,.
\]

\subsubsection{} Put
\[
h(t) = \max_{0 \le k \le n} \bigl( \log |\la_k| - 2\pi k t \bigr)\,, \qquad H(t) = e^{h(t)}\,.
\]
By $\nu (t)$ we denote {\em the central index}, that is, the largest of the indices $\nu$, for which
\[
\log |\la_\nu| - 2\pi \nu t \ge \log |\la_{k}| - 2\pi k t\,, \qquad k\in\{0, 1, 2, \,...\, n \}\,.
\]

Obviously, $H(t) \le S(e^{-2\pi t}; P) \le (n+1) H(t)$. This will allow us to replace $S$ by $H$ in our estimates.
The advantage of $H$ over $S$ is that the former has sharper transitions at the points where the central index changes
its value.

\subsubsection{}\label{subsub:lower-bound}
In the new notation, Theorem~\ref{thm:lower-bound} says that {\em given $t\ge 0$, given an interval
$I=[s', s'']$ of length less than $1$, and given a positive integer parameter $m$,
there exists an event
\[
\Sigma (t, I, m) \subset \{+,-\}^{n+1}\,, \qquad {\rm with} \quad \cP \bigl( \Sigma (t, I, m) \bigr) \le 2^{-m}
\]
such that for every $ \sigma\in \{+, -\}^{n+1}\setminus\Sigma (t, I, m)  $,}
\begin{equation}\label{eq:lower-bd}
\max_{s\in I} \bigl| Q(t+{\im} s) \bigr| \ge n^{-2} (c\, |I| )^{6m} H(t)\,.
\end{equation}
This estimate is complemented by the obvious upper bound
\begin{equation}\label{eq:upper-bd}
\max_s \bigl| Q(t+{\im} s) \bigr| \le (n+1) H(t)\,.
\end{equation}

\subsection{The test sets and exceptional sign sequences}
Our exceptional event $\Sigma\subset \bigl\{ +, - \bigr\}^{n+1}$ will be a union of the events
$ \Sigma (t, I, m) $ taken over a certain finite sets of ``test points'' $t$ and ``test intervals'' $I$.
So we start by defining these sets.

\subsubsection{}
Recall that each $\la_k$ attains two values, consider the $2n+2$ lines
$  t\mapsto \log|\la_k^{\pm}|-2\pi k t $, $0 \le k \le n $.
There are at most $ {2n+2 \choose 2} = (n+1)(2n+1) $ points on $[0, \infty)$ where two of these functions are equal.
We denote this set of points by $\mathfrak T_0$. Then we put
\[
\mathfrak T = \Bigl\{ t= \frac{j}{n}\colon {\rm dist\ }(t, \mathfrak T_0) \le 1, \ j\in \bZ_+ \Bigr\}\,.
\]
This will be our set of {\em test points} $t$.

Put
\[
\mathfrak S = \Bigl\{ s=\frac{k}{n}\colon 0\le k < n, \ k\in \bZ_+ \Bigr\}\,.
\]
The set $\mathfrak I$ of {\em test intervals} $I$ will consist of all
intervals centered at all the points $s\in\mathfrak S$, of length $j/n$ with $1\le j \le n$, $j\in \bN$.

\begin{claim}\label{claim:test-pairs}
The cardinality of the set $\mathfrak T$ is $ \le C n^3 $.
The cardinality of the set $\mathfrak J$ is $\le Cn^2$.
\end{claim}

\begin{proof} Obvious. \end{proof}

\subsubsection{}
Now, we define the exceptional event $ \Sigma \subset \bigl\{ +, - \bigr\}^{n+1}$. Put
\[
\Sigma (m) = \bigcup_{t\in\mathfrak T, \ I\in\mathfrak I}
\Sigma (t, I, m)
\]
where the events $\Sigma (t, I, m)$ are the same as in~\ref{subsub:lower-bound}.
Then, by the last claim,  $ \cP \bigl( \Sigma (m) \bigr) \le C n^5\, 2^{-m} $.
Now, we fix
\[
m = C(A) \log n
\]
with a sufficiently large value $C(A)$, and let $\Sigma = \Sigma (m)$. Then
$ \cP \bigl( \Sigma \bigl) < n^{-A}$.

\medskip
In the rest of the proof, we fix the sign sequence $\sigma\in\{+, -\}^{n+1}\setminus\Sigma$.
We put $V=V(P)$, where $V(P)$ is the number of vertices on the Newton-Hadamard polygon
introduced in Section~\ref{subsubsub:Newton-Hadamard}.

\subsection{The Whitney-type partition}

\subsubsection{}
For $t\ge 0$, the graph of the function $h(t)$ is a piece-wise linear function with
at most $V+1$ intervals of linearity. Take one of these intervals and call it $\mathbb J$.
The proof of Theorem~\ref{thm:main} needs a special partition of the interval $\mathbb J$.
To construct this partition, we take $L$ as in Theorem~\ref{thm:main}, let $L' = [L]+4$ (as usual, $[L]$ stands for the integer part
of $L$) and take a sequence of closed intervals with disjoint interiors starting in both directions from
the test points in $\mathfrak T \setminus \operatorname{int}(\mathbb J)$ closest to the end points
of $\mathbb J$ so that
\begin{itemize}
\item
the end-points of each interval of this sequence belong to the set $\mathfrak T$ of test-points;
\item
the first $4L'$ intervals starting with each end-point
of $\mathbb J$ have length $\frac2{n}$, the next $4L'$ intervals have length $\frac4{n}$, the next $4L'$
have the length $\frac8{n}$, and so on, until we either
reach length $1$ or cover the middle point of $\mathbb J$
\end{itemize}
(see Fig.~1).
We denote the intervals of this sequence by $J$ and note that
we used at most $CL'\log n$ intervals $J$ per each interval $\mathbb J$.
\begin{sidewaysfigure}[h]
\begin{center}
	\scalebox{1.6}{
\begin{tikzpicture}

  \begin{scope}[shift={(-2.5,0)}]
	  \draw [->] (-4.5,0) -- (7.0,0);
\foreach \x in {-3.75,-3.5,-3.1,-2.45,-1.5,
0.25,1.25,1.75,2,2.25}
\draw (\x,1pt)--(\x,-1pt);

\draw [-] (-4.25,4) -- (-3.65,3.5) node[right] {\tiny \hskip 5pt Graph of $h$};
\draw [-] (-3.65,3.5) -- (2.1,2);
\draw [-] (2.1,2) -- (5, 1.75);

\draw [dashed] (-3.65,0) -- (-3.65,3.5);
\coordinate (A) at (-3.65,0);
    \filldraw (A) circle (1pt) node[below] {\tiny \hskip 15pt $u\in {\mathfrak T}_0$};
\coordinate (AA) at (-1.9,0);
    \filldraw (AA) circle (1pt) node[below] {\tiny $u+1$};
\draw [dashed] (2.1,0) -- (2.1,2);
\coordinate (B) at (2.1,0);
    \filldraw (B) circle (1pt) node[below] {\tiny $v\in {\mathfrak T}_0$};
\coordinate (BB) at (0.35,0);
    \filldraw (BB) circle (1pt) node[below] {\tiny $v-1$};
    \draw [thick,
    decorate, decoration={brace,mirror}] (-1.5,-0.35) -- node[below]
    {\tiny untested region } (0.25,-0.35);
	  \draw [thick, -] (-1.5,0) -- (0.25,0);
    \draw [thick,
    decorate, decoration={brace}] (-3.75,0.15) -- node[above]
    {\tiny left partition of $\mathbb J$ } (-1.5,0.15);
    \draw [thick,
    decorate, decoration={brace}] (0.25,0.15) -- node[above]
    {\tiny right partition of $\mathbb J$ } (2.25,0.15);

    \draw[thick] (-2.7,0) circle (1.3);
    \draw[dashed] (-3.5,-1) -- (-4.5,-2.2);
    \draw[dashed] (-1.9,-1) -- (8.5,-2.2);

   \draw (-4.5,-2.2) .. controls (0,-1.5) .. (8.5,-2.2);
\draw (8.5,-2.2) .. controls (9,-3.5)  .. (8.5,-5);
\draw (8.5,-5) .. controls (0,-6) .. (-4.5, -5);
\draw (-4.5,-5) .. controls (-5,-3.5) .. (-4.5, -2.2);
  \end{scope}


  \begin{scope}[shift={(-2.5,-3)}]
	  \draw [->] (-4.5,0) -- (7.0,0) node[right] {$\cdots$ $\cdots$};
    \draw[<->] (-4.5,0.35) -- node[above] {\tiny $1/n$}
    (-4.25,0.35);
    \draw[<->] (-4,0.35) -- node[above] {\tiny $2/n$}
    (-3.5,0.35);
    \draw[thick, -] (-4,0) --
    (-3.5,0);
    \draw[<->] (-2.5,0.35) -- node[above] {\tiny $4/n$}
    (-1.5,0.35);
    \draw[thick, -] (-2.5,0) --
    (-1.5,0);
    \draw[<->] (1.5,0.35) -- node[above] {\tiny $8/n$}
    (4,0.35);
    \draw[thick, -] (1.5,0) --
    (4,0);
    \draw [thick,
    decorate, decoration={brace, mirror}] (-4.5,-0.45) -- node[below]
    { \tiny $4L'$ intervals } (-2.6,-0.45);
    \draw (-3.5,-0.85) node[below] {\tiny of length $2/n$};
    \draw [thick,
    decorate, decoration={brace, mirror}] (-2.4,-0.45) -- node[below]
    {\tiny $4L'$ intervals } (1.4,-0.45);
    \draw (-0.5,-0.85) node[below] {\tiny of length $4/n$};
    \draw [thick,
    decorate, decoration={brace, mirror}] (1.6,-0.45) -- node[below]
    {\tiny  $4L'$ intervals } (8.5,-0.45);
    \draw (5.1,-0.85) node[below] {\tiny of length $8/n$};
    \coordinate (A) at (-4.35,0);
    \filldraw (A) circle (1pt) node[below] {\tiny
	    $u$};

    \foreach \x in {-4.25,-3.75,-3.25,-2.75,-2.25,-2,-1.75,-1.25,
	-1,-0.75,-0.25,0,0.25,0.75,1,1.25,1.75,2,2.25,2.5,
2.75,3, 3.25,3.5,3.75,4.25,4.5,4.75,5,5.25,5.5,5.75,6,6.25,6.75}
\draw (\x,1pt)--(\x,-1pt);
\foreach \x in {-4,-3.5,-3}
\draw (\x,4pt)--(\x,-4pt);
\foreach \x in {-4.5,-2.5, 1.5, 4,6.5}
\draw (\x,8pt)--(\x,-8pt);

\foreach \x in {-1.5,-0.5,0.5}
\draw (\x,5pt)--(\x,-5pt);

  \end{scope}

\end{tikzpicture}
}
{\caption{Partition of the interval $\mathbb J=[u, v]$ with $u, v\in\mathfrak T_0$;
test points are within distance $1/n$ of each other. This figure corresponds to
the (impossible) value $L'=1$.}}
\label{fig:WhitneyPartition}
\end{center}
\end{sidewaysfigure}
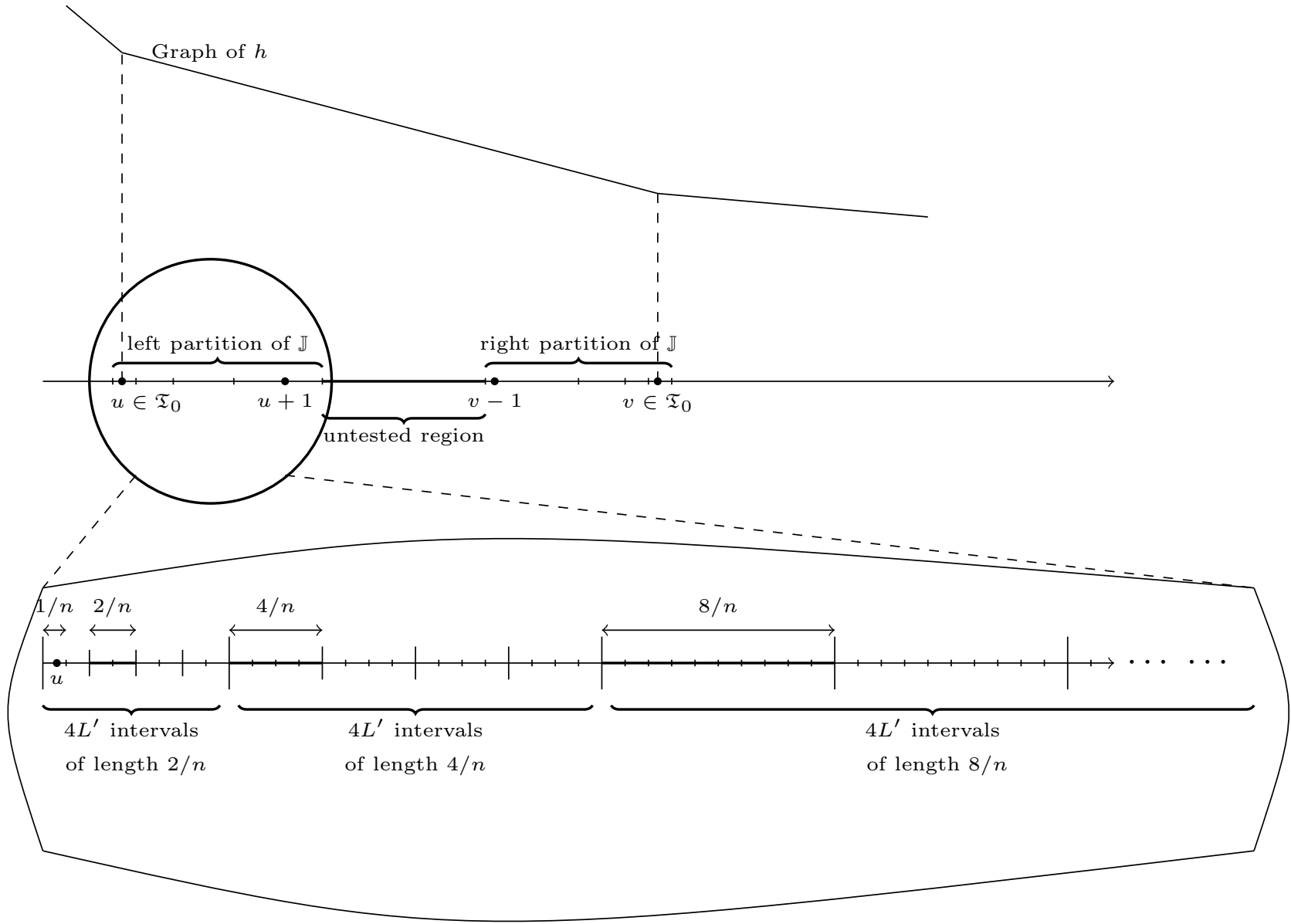

Next, we list several properties of this construction, which will be used in the
proof of Theorem~\ref{thm:main}.

\subsubsection{}\label{subsubsect:tested-pairs}
The centers $c_J$ of the intervals $J$ belong to the set $\mathfrak T$ of tested points.
The intervals $I=\bigl[ s - \frac12 |J|, s + \frac12 |J| \bigr]$, $s\in \mathfrak S$,
belong to the set $\mathfrak I$ of tested intervals.

\subsubsection{}\label{subsubsect:J'}
By $J'$ we denote the interval centered at $c_J$ which is $L'$ times longer than $J$. Then, by construction,
if $J$ is an interval from our partition with $|J|\ge \frac4{n}$, then $J'\subset \mathbb J$.

\subsection{There are no zeroes in the strips with the untested ground}

By $\mathbb J_0$ we denote the part of $\mathbb J$ that remains uncovered by intervals $J$ and call it
{\em the untested part of $\mathbb J$}. For some intervals $\mathbb J$, the untested part $\mathbb J_0$
can be void. Let $\Pi_{\mathbb J_0} = \bigl\{t+\im s\colon t\in \mathbb J_0 \bigr\}$ be the corresponding
vertical strip.

\begin{claim}\label{claim:untested}
The exponential polynomial $Q$ does not vanish on all vertical strips
$\Pi_{\mathbb J_0} $.
\end{claim}

\noindent{\em Proof of Claim~\ref{claim:untested}}:
Suppose that the point $t$ belongs to one of the intervals $\mathbb J_0$,
that is, the central index $\nu$ stays fixed on $[t-1, t+1]$.
Thus, we actually have not only
\begin{align*}
\log |\la_\nu| - 2\pi \nu t &\ge \log |\la_k| - 2\pi k t\,, \\
\intertext{but also}
\log |\la_\nu| - 2\pi \nu t &\ge \log |\la_k| - 2\pi k t + 2\pi |k-\nu|\,.
\end{align*}
Then
\begin{align*}
\Bigl| \sum_{k\ne \nu} \la_k e^{-2\pi k t} \Bigr|
&\le \sum_{k\ne\nu} |\la_k| e^{-2\pi k t} \\
&\le |\la_\nu| e^{-2\pi \nu t}\, \sum_{k\ne\nu} e^{-2\pi |k-\nu|}
= |\la_\nu| e^{-2\pi \nu t} \cdot 2 \sum_{k\ge 1} e^{-2\pi k} < |\la_\nu| e^{-2\pi \nu t}\,.
\end{align*}
Hence, $Q$ cannot vanish on the vertical line $t+\im\bR$, and therefore, on the whole vertical strip
$\Pi_{\mathbb J_0}$. \hfill $\Box$

\subsection{Jensen's bound for the number of zeroes of $Q$ in the disks $\bar D_{J, s}$}

Given interval $J$ from our partition and $s\in\mathfrak S$, consider the disks
\[
D_{J, s}=\bigl\{w\colon |w-(c_J+{\rm i}s)|<\tfrac12 L'|J| \bigr\},
\]
and denote by $N(\bar D_{J, s}; Q)$ the number of zeroes of $Q$ in the closed disk
$ \bar D_{J, s} $ counted with multiplicities.

\begin{claim}\label{claim:Jensen}
We have
\[
N(\bar D_{J, s}; Q) \le C(A, L) \log^2 n\,.
\]
\end{claim}

\medskip
The proof of this claim relies upon ``the classical Jensen's bound''\footnote{
For the reader's convenience, we recall its short proof, assuming, without lost of generality, that $D$ is the unit disk.
Let $a_1$, ..., $a_N$ be zeroes of $F$ in $\frac12 \bar D$ counted with multiplicities, and let
\[
B_a(z) = \frac{z-a}{1-z\bar a}\,.
\]
Then $F=B_{a_1}\,...\, B_{a_N} G$, where the function $G$ is analytic in $D$
and $\sup_D |F| = \sup_D |G|$. Note that the absolute value of each factor $B_{a_i}$ is bounded by $\frac45$ in $\frac12 \bar D$.
Therefore,
\[
\max_{\frac12 \bar D} |F| \le \bigl( \tfrac45 \bigr)^N \max_{\frac12 \bar D} |G| \le \bigl( \tfrac45 \bigr)^N \sup_{D} |G| =
\bigl( \tfrac45 \bigr)^N \sup_{D} |F|\,,
\]
whence, the estimate. \hfill $\Box$
}:
{\em Let $F$ be an analytic function in a disk $D$. Let
$\tfrac12 D$ be the disk concentric with $D$ but of twice smaller radius.
Then the number of zeroes of $F$ in the closed disk $\tfrac12 \bar D$ (counted with multiplicities)
is}
\[
\le C \log \frac{\sup_D |F|}{\max_{\frac12 \bar D} |F|}\,.
\]

\medskip\noindent{\em Proof of Claim~\ref{claim:Jensen}}:
By~\ref{subsubsect:J'}, the intervals $J\subset\mathbb J$ fall into two categories:
either $J' \subset \mathbb J$, or the length of $J$ is $\frac2{n}$.
First, consider the intervals $J$ from the first group, that is, assume that
the central index $\nu$ stays fixed on $J'$. Take the function
$F(w)=Q(w)e^{2\pi\nu w}$ which has the same zeroes as $Q$.
By~\ref{subsubsect:tested-pairs}, each point $c_J$ and each interval $I_{J, s}=\bigl[ s-\tfrac12|J|, s+\tfrac12 |J|\bigr]$
are tested. Note that $c_J +\im I_{J, s} \subset D_{J, s}$. Therefore, we have the lower bound
\[
\max_{\bar D_{J,s}} |F| \stackrel{\eqref{eq:lower-bd}}\ge
\max_{v\in I_{J,s}} |F(c_J+\im v)| \ge n^{-2}\, \bigl( c\, |I_{J, s}| \bigr)^{6m} H(c_J) \cdot e^{2\pi \nu c_J}
=  n^{-2}\, \bigl( c\, |J| \bigr)^{6m} |\la_\nu|\,.
\]
The matching upper bound
\[
\max_{t+\im v\in 2\bar D_{J, s}} |F(t+\im v)| \stackrel{\eqref{eq:upper-bd}}\le (n+1) \max_{t\in J'} [ H(t)  e^{2\pi \nu t} ] < 2n |\la_\nu|
\]
is evident. Using Jensen's bound, recalling that $|J|\ge\frac4{n}$ and that
$m=C(A)\log n$, we get
\[
N(\bar D_J; F) \le C m \log n \le C(A) \log^2 n\,.
\]

\medskip
Now, we turn to the second case, when $|J|=\frac2{n}$. These intervals are so short that
the function $h$ can change only by a constant (depending on $L'$) on $J'$. Indeed, let $J'=[a, b]$.
Take the points $a=t_0 < t_1 <\,...\,<t_s=b$, so that
$ \nu (t_i+0) = \nu (t_{i+1}-0) = \nu_i $.
Then
\[
h(a) - h(b) = \sum_{i=0}^{s-1} [ h(t_i) - h(t_{i+1})] =  \sum_{i=0}^{s-1} 2\pi \nu_i [t_{i+1}-t_i] \le 2\pi n (b-a) = 4 \pi L'\,;
\]
the estimate in the opposite direction is obvious since the function $h$ does not increase.
Therefore,
$ H(a) / H(b) \le e^{4\pi (L+4)} $.

Then, similarly to the first case, we take the corresponding test intervals $I_{J,s}$, note that
\[
\max_{\bar D_{J, s}} |Q| \ge \max_{I_{J, s}} |Q| \stackrel{\eqref{eq:lower-bd}}\ge n^{-2} (c|I_J|)^{6m} H(c_J)\,,
\]
and that
\[
\max_{2\bar D_{J, s}} |Q| \stackrel{\eqref{eq:upper-bd}}\le (n+1) \max_{J'} H < 2n\, e^{4\pi (L+4)} H(c_J)\,.
\]
Then, applying Jensen's bound to the function $Q$, and recalling that $|I_{J, s}|=|J|= \frac2{n}$ and that $m=C(A) \log n$,
we get $N(\bar D_{J, s}, Q) \le C(A, L) \log^2 n$. This proves Claim~\ref{claim:Jensen}. \hfill $\Box$

\subsection{Completing the proof of Theorem~\ref{thm:main}}

Take an arbitrary $L$-Lipschitz curve
$  \Gamma = \bigl\{t+{\rm i}s(t)\colon 0 \le t <\infty \bigr\} $, $|s(t_1) - s(t_2)|\le L|t_1-t_2|$, and
let $\Gamma_K = \bigl\{ t+\im s(t)\colon t\in K \bigr\}$ be the part of $\Gamma$ that lies over an interval $K$.
Since the curve $\Gamma$ is $L$-Lipschitz and $c_J+{\rm i}s(c_J)\in \Gamma_J$, we see that $\Gamma_J$ does not exit the
rectangle
\[
\bigl\{ t+ \im v\colon t\in J,\, |v-s(c_J)|\le \tfrac12 L|J| \bigr\}\,
\]
Let $s_J'$ be a point in $\mathfrak S$ closest to $s(c_J)$
(if there are two such points, choose any of them). Put $D_J=D_{J, s_J'}$.
Then, $\Gamma_J \subset \bar D_{J}$. Therefore,
\[
\Gamma_\mathbb J \setminus \Pi_{\mathbb J_0}  = \bigcup_{J\subset \mathbb J} \Gamma_J
\subset \bigcup_{J\subset \mathbb J} \bar D_J\,.
\]

By Claim~\ref{claim:untested},
$Q$ does not vanish in the vertical strips $\Pi_{\mathbb J_0}$ generated
by the untested parts $\mathbb J_0$.
Therefore,
\[
N(\Gamma; Q) \le \sum_{\mathbb J}\, \sum_{J\subset\mathbb J} N(\bar D_J; Q)\,,
\]
where $N(\Gamma; Q)$ is the number of zeroes of $Q$ on $\Gamma$.

By Claim~\ref{claim:Jensen}, $N(\bar D_J; Q) \le C(A, L)\log^2 n$.

At last, recall that the number of intervals $J$ used per each interval
$\mathbb J$ is at most $C L \log n $, and that the number of the
intervals $\mathbb J$ where the central index $\nu$ stays fixed is at most $V+1\le 2V$.
All together, this gives us
\[
N(\Gamma; Q) \le C(A, L) V \log^3 n\,,
\]
completing the proof of Theorem~\ref{thm:main}.
\hfill $\Box$


\begin{thebibliography}{A}

\bibitem{BP} {\sc A. Bloch, G. P\'olya},
On the roots of certain algebraic equations.
Proc. London Math. Soc. {\bf 33} (1932), 102--114.

\bibitem{EO} {\sc P. Erd\H{o}s, A. C. Offord},
On the number of real roots of a random algebraic equation.
Proc. London Math. Soc. (3) {\bf 6} (1956), 139--160.


\bibitem{IZ} {\sc I. Ibragimov, Dm. Zaporozhets},
On distribution of zeros of random polynomials in complex plane.
In: Prokhorov and contemporary probability theory, 303--323,
Springer Proc. Math. Stat., {\bf 33}, Springer, Heidelberg, 2013.

\bibitem{IM} {\sc I. A. Ibragimov, N. B. Maslova},
The mean number of real zeros of random polynomials. \\
I. Coefficients with zero mean (Russian).
Teor. Verojatnost. i Primenen. {\bf 16} (1971) 229--248; English transl. in Theor. Probability Appl. {\bf 16} (1971), 228--248; \\
II. Coefficients with a nonzero mean, ibid, 495--503; English transl. 485--493.

\bibitem{IM-Doklady} {\sc I. A. Ibragimov, N. B. Maslova},
The average number of real roots of random polynomials. (Russian)
Dokl. Akad. Nauk SSSR {\bf 199} (1971), 13--16;
English transl. in Soviet Math. Dokl. {\bf 12} (1971), 1004--1008.

\bibitem{Kac} {\sc M. Kac},
On the average number of real roots of a random algebraic equation.
Bull. Amer. Math. Soc. {\bf 49} (1943), 314--320;
A correction, ibid, 938.

\bibitem{Kolm} {\sc A. N. Kolmogorov},
Sure les propri\'et\'es des fonctions de concentrations de M. P.~L\'evy.
Ann. Inst. H.~Poincar\'e {\bf 16} (1958), 27--34.

\bibitem{LO} {\sc J. E. Littlewood, A. C. Offord},
On the number of real roots of a random algebraic equation.\\
I. Journal London Math. Soc. {\bf 13} (1938), 288--295;\\
II. Proc. Cambridge Phil Soc. {\bf 35} (1939), 133--148; \\
III. Rec. Math. [Mat. Sbornik] N.S. {\bf 12(54)} (1943). 277--286.

\bibitem{LS} {\sc B. F. Logan, L. A. Shepp},
Real zeros of random polynomials. \\
{\rm I}, Proc. London Math. Soc. (3) {\bf 18} (1968), 29--35; \\
{\rm II}, ibid, 308--314.

\bibitem{Mont} {\sc H. L. Montgomery},
Ten Lectures on the Interface Between Analytic Number Theory and Harmonic Analysis.
Amer. Math. Soc., 1994.

\bibitem{Nazarov} {\sc F. Nazarov},
Local estimates for exponential polynomials and their applications to inequalities of the uncertainty principle type. (Russian)
Algebra \& Analiz {\bf 5} (1993), no. 4, 3--66; English translation in
St. Petersburg Math. J. {\bf 5} (1994), 663--717.

\bibitem{PolyaSzego}
{\sc G. P{\'o}lya, G. Szeg\H{o}},
{Problems and Theorems in Analysis {II}}.
Reprint of the 1976 English translation. Classics in Mathematics. Springer-Verlag, Berlin, 1998.

\bibitem{Soze2} {\sc K. S\"oze},
Real zeroes of random polynomials, II.
Descartes' rule of signs and anti-concentration
on the symmetric group.

\bibitem{SF} {\sc L. Shepp, K. Farahmand},
Expected number of real zeros of a random polynomial with independent identically distributed symmetric long-tailed coefficients.
Teor. Veroyatn. Primen. {\bf 55} (2010), 196--204;
Theory Probab. Appl. {\bf 55} (2011), 173--181.

\bibitem{Zaporozhets} {\sc D. N. Zaporozhets},
An example of a random polynomial with unusual behavior of the roots. (Russian)
Teor. Veroyatn. Primen. {\bf 50} (2005), 549--555; English translation in
Theory Probab. Appl. {\bf 50} (2006), 529--535.


\end{thebibliography}
\end{document}